\documentclass[12pt,fleqn]{article}
\usepackage{amsmath,amsfonts,amssymb,amsthm,color,graphicx,graphics}
\usepackage{appendix}
\usepackage[colon,authoryear]{natbib}
\newtheorem{thm}{Theorem}[section]

\newtheorem{lem}[thm]{Lemma}
\numberwithin{equation}{section}

\title{Theoretical Study of Pest Control Using Stage Structured Natural Enemies with Maturation Delay: A Crop-Pest-Natural Enemy Model}
\author{Kunwer Singh Jatav \and Joydip Dhar}
\author{{\bf Kunwer Singh Jatav$^*$,  Joydip Dhar}\\[0.1in] Department of
Applied Sciences,\\ ABV-Indian Institute of Information Technology
and Management
\\ Gwalior(M.P.)-474015, INDIA\\ E-mail: $^*sing1709@gmail.com$, $jdhar@iiitm.ac.in$}
\begin{document}
\maketitle
\begin{abstract}
In the natural world, there are many insect species whose individual members have a life history that
takes them through two stages, immature and mature. Moreover, the rates of survival, development, and reproduction almost always depend on age, size, or development stage. Keeping this in mind, in this paper, a three species crop-pest-natural enemy food chain model with two stages for natural enemies is investigated. Using characteristic equations, a set of sufficient conditions for local asymptotic stability of all the feasible equilibria is obtained.  Moreover, using approach as in \citep{BerettaKuang2002}, the possibility of the existence of a Hopf bifurcation for the interior equilibrium with respect to maturation delay is explored, which shows that the maturation delay plays an important role in the dynamical behavior of three species system. Also obtain some threshold values of maturation delay for the stability-switching of the particular system. In succession, using the normal form theory and center
manifold argument, we derive the explicit formulas which determine the stability and direction of bifurcating
periodic solutions. Finally, a numerical simulation for supporting the theoretical analysis is given.
\end{abstract}
\section{Introduction}\label{introduction}
  It is a well known fact that pest is a harmful insect and its outbreak often cause serious ecological and economic problems \citep{kaminska2004geographical,weaver1992pesticide}. Evidence indicates that annually the pests cause 25\% loss in rice, 5-10\%
 in wheat, 30\% in pulses, 35\% in oilseeds, 20\% in sugarcane and 50\% in cotton \citep{Dhaliwal1996}. Now a days, many pest control methods are available, such as biological, cultural, physical and chemical methods \citep{franz1961biological,van1988biological,vincent2003management,hoyt1969integrated}. However farmers mostly use pesticides to control pests because of its efficiency and convenience. The chemical pesticide kills not only pests but it also kills their natural
enemies. Actually, when pests are caught or poisoned to a large extent, their natural enemies become extinct due to no food and afterwards, when the pesticide intensity decreases, then the pests increase rapidly. Due to this reason, chemical control has become challenged. Furthermore,
the common practice proves that long-term adopting chemical control may give rise to disastrous results, for example, environmental
contamination, toxicosis of the man and animals and so on. Thus the pesticide pollution is also recognized as a major health hazard to human beings and to natural enemies \citep{butler1969monitoring,dahal1995study,kaminska2004geographical,kishimba2004status}. On the other hand, it is a well known that the biological control method is harmless to human, animal and environment. Biological control is generally used to control a particular pest using a chosen living organism; this chosen organism might be a predator, parasite or disease which attacks on the harmful insect pest. The last few years have seen an sudden increase of interest in the study of biological pests control using prey-predator interaction \citep{Jiao2008a,Liu2003,dong2006extinction,Liu1995,WWang2001,Song2006,RShiLChen2009}.
\par Moreover, effective use of biological control often requires a good understanding of the biology of the pest, its natural
enemies and their interaction, as well as the ability to identify various life stages of relevant insects in the crops. Again, the formulations of stage structured population models are the recognition that individuals of many species have life-histories composed of a sequence of stages within which their
characteristics are broadly similar to those of other individuals in the same stage and totally
different, from those of individuals in other stages. In insect population, such stages are particularly easy
to recognize, being separated by short events such as moult or pupation. Many researchers studied stage-structured models before 1990 \citep{BarclayDriessche1980,BenceNisbet1989,GurneyNisbetLawton1983,GurneyNisbet1985,Hasting1983,LandahlHanson1975,WoodBlytheGurneyNisbet1989}, but the real interest comes into picture on the stage structured
models after the work of Aiello and Freedman \citep{AielloFreedman1990}. They proposed a single species model with stage structure assuming an average age to maturity (i.e., as a constant time delay) which reflecting a delayed birth of immature and a reduced survival of immature to their maturity. The model is as follows:
\[\frac{dx(t)}{dt}=\beta y(t)-rx(t)-\beta e^{-r\tau} y(t-\tau),\]
\begin{equation}
\frac{dy(t)}{dt}=\beta e^{-r\tau} y(t-\tau)-\eta y^2(t),
\label{eq1.1}
\end{equation}
where $x(t)$ and $y(t)$ represent the immature and mature populations densities, respectively. Here, it is  assumed that at any time $t>0$, growth rate of immature population is proportional to the existing mature population with proportionality constant $\beta$; the death rate of immature population is $r$; the death rate of mature population is proportional to the square of the population with the proportionality constant $\eta$. The term $\beta e^{-r\tau} y(t-\tau)$ represents the immature who were born at time $t-\tau$ and survive at time $t$, therefore it represents the transformation of immature to mature, where $\tau$ represent a constant time to maturity. All the parameters $\tau$, $\beta$, $r$ and $\eta$ are positive constants.
\par Further, the single species model (\ref{eq1.1}) is extended by many researchers into different kinds of stage-structured models and obtained significant results \citep{AielloFreedmanWu1952,CaoFanGard1992,FreedmanWu1991,FreedmanSoWu,HuoLiAgarwal2001,KGMagnusson1999}. Recently, many authors studied different kinds of predator-prey system with division of the predators into immature and mature class and a good number of research has been carried out \citep{WWang2001,RXu2004,SGao2008,KGMagnusson1999,SunHuoXiang2009,QuWei2007,HuHuang2010,SatioTakeuchi2003}. One interested model is suggested by Satio and Takeuchi \citep{SatioTakeuchi2003}. They considered two life stages for predator and proposed the following predator-prey model:
\[ \dot{x}(t)=x(t)\left(r_1-a_{11}x(t)-a_{13}y(t)\right),\]
\begin{equation}
\dot{Y}(t)=-r_2 Y(t)+a_{31} x(t)y(t)-a_{31} e^{-r_2 \tau}x(t-\tau)y(t-\tau),
\label{eq1.2}\end{equation}
\[ \dot{y}(t)=-r_3y^2(t)+a_{31} e^{-r_2 \tau}x(t-\tau)y(t-\tau),\]
where $x(t)$ is population density of prey, $Y(t)$ and $y(t)$ denote the densities of immature and mature predator population, respectively; $\tau$ represent a constant time to maturity for predator; $a_{13}$ is the per capita rate of predation; $a_{31}$ is the conversion rate and all other parameters have the similar meaning as in (\ref{eq1.1}).
\par Furthermore, three species food chain models are investigated by many researchers \citep{FreedmanWaltman1977,FreedmanSo1985,Hasting1991,FreedmanRuan1992,McCannYodzis1995,BoerKooiKooijman1999,Kuang2000,HsuaHwang2003}. One of noteworthy contribution is given by Kuang et al. \citep{HsuaHwang2003}, they considered a three trophic food chain model for plant-pest-natural enemy, but they ignored the stage structure phenomena of species.
\par The aim of this paper is to study a crop-pest-natural enemy model with two life stages of natural enemy. The paper is organized as follows: in section \ref{model}, model development is discussed, the positivity and boundedness are established in the section \ref{positivityboundedness}. In section \ref{stability}, all the feasible equilibria and their local stability behavior are studied. The stability and direction of Hopf bifurcation is analyzed in section \ref{direction}. Further, in section \ref{numerical}, a set of numerical simulations is given to verify all the major analytical findings. Finally, conclusions for this paper are given in the last section.

\section{Proposed Mathematical Model } \label{model}
In this section, our main aim is to propose a mathematical model for the interaction of plant-pest-natural enemy. Since plant-hoppers are serious pests for the rice crops and these are  suppressed by Lycosa tarantula and other spiders, it is well documented in a report of the Indian Council of Agricultural Research (ICAR) \citep{Birthal2004}. The tarantula species has two major life stages, namely, immature and mature; only mature population can harvest the pest and reproduce a new offspring. In modelling process, we assume that $x(t)$, $y(t)$, $z_1(t)$ and $z_2(t)$ are densities of crop, pest, immature and mature natural enemy at time $t$, respectively. The parameters $a_1$, $b_1$ are respectively the intrinsic growth rate and overcrowding rate of crops; $c_1$, $\alpha_1$ are per capita predation rate of crop by the pest and the corresponding growth rate of pest, respectively. The parameters  $c_2$ is per capita predation rate of pest by the natural enemy and $\alpha_2$ is the corresponding growth rate of mature natural enemy. Here, $d_1,\ d_2, \ d_3$ are the natural death rates of pests, immature and mature natural enemies, respectively. Further, $\tau$ is the maturation delay from immature to mature natural enemies, the term $\alpha_2 e^{-d_2 \tau}y(t-\tau)z_2(t-\tau)$ represents the transformation of immature to mature population. Keeping this biological situation  in mind and motivated from the modelling ideas of \citep{AielloFreedman1990,SatioTakeuchi2003}, in this paper, we propose a three species stage structured crop-pest-natural enemy model as follows:
\[ \frac{dx(t)}{dt}=x(t)\left(a_1-b_1x(t)-c_1y(t) \right),\]
\[ \frac{dy(t)}{dt}=y(t)\left(\alpha_1x(t)-d_1-c_2z_2(t)\right),\]
\begin{equation}
 \frac{dz_1(t)}{dt}=\alpha_2y(t)z_2(t)-d_2z_1(t)-\alpha_2e^{-d_2\tau}y(t-\tau)z_2(t-\tau),
\label{eq2.1}
\end{equation}
\[\frac{dz_2(t)}{dt}=\alpha_2e^{-d_2\tau}y(t-\tau)z_2(t-\tau)-d_3z_2(t),\]
The model completes with the following set of initial conditions:
\[ x(\theta)=\phi_1(\theta),\ y(\theta)=\phi_2(\theta), \ z_1(\theta)=\psi_1(\theta),\]
\begin{equation}
z_2(\theta)=\psi_2(\theta),\ \phi_i(0)>0, \ \psi_i(0)>0,\ \ \theta\in [-\tau,0],\ \ i=1,2,\label{eq2.3}
\end{equation}
where $(\phi_1,\phi_2,\psi_1,\psi_2)\in C \left([-\tau,0],R_{+}^4\right)$, the Banach space of continuous functions mapping on the interval $\left[-\tau,0\right]$ into $R_{+}^4$. For continuity of the initial conditions, we further require
\begin{equation}
\psi_1(0)=\int_{-\tau}^{0} \alpha_2\phi_2(s)\psi_2(s)e^{d_2s}ds,\label{eq2.4}
\end{equation}
where $\psi_1(0)$ represents the accumulated survivors of those natural enemy members who were born between $-\tau$ and $0$.
In the next section, we will discuss the positivity and boundedness of solutions of the system (\ref{eq2.1}) with initial conditions (\ref{eq2.3}) and (\ref{eq2.4}).
\section{Positivity and boundedness}\label{positivityboundedness}
Positivity means that the species is persistent and boundedness implies a natural restriction. Therefore, we can state and prove the following lemmas for the positivity and boundedness of solutions of the system (\ref{eq2.1}):
\begin{lem}\label{positivity}
The solutions of system (\ref{eq2.1}) with initial conditions (\ref{eq2.3}) and (\ref{eq2.4}) are positive, for all $t\geq 0$.
\end{lem}
\begin{proof}
Let $\left(x(t),y(t),z_1(t),z_2(t)\right)$ be a solution of system (\ref{eq2.1}) with initial conditions (\ref{eq2.3}) and (\ref{eq2.4}). Let us first consider $z_2(t)$ for $t\in[0,\tau]$. Noting that $\phi_2(\theta)\geq 0$, $\psi_2(\theta)\geq 0$ for $\theta \in [-\tau,0]$, we obtain from the fourth equation of system (\ref{eq2.1}) that
\begin{equation}
\frac{dz_2(t)}{dt}=\alpha_2e^{-d_2\tau}\phi_2(t-\tau)\psi_2(t-\tau)-d_3z_2(t)\geq -d_3z_2(t).
\nonumber\end{equation}
It thus follows for $t\in [0,\tau]$,
\begin{equation}
z_2(t)\geq z_2(0)e^{-d_3t}>0.
\nonumber\end{equation}
For $t\in [0,\tau]$, the second equation of the system (\ref{eq2.1}) can be rewritten as
\begin{equation}
\frac{dy(t)}{dt}\geq y(t)\left(-d_1-c_1z_2(t)\right).\nonumber\\
\end{equation}
A standard comparison argument shows that for $t\in [0,\tau]$,
\begin{equation}
y(t)\geq y(0)\exp\left( \int_0^t\left(-d_1 -c_1z_2(s)ds\right)\right)>0.\nonumber\end{equation}
Similarly, it follows from the first equation of system (\ref{eq2.1}) that for $t\in [0,\tau]$,
\begin{equation}
\frac{dx(t)}{dt}\geq -b_1x^2(t)-c_1x(t)y(t),\nonumber
\end{equation}
which evidences that
\begin{equation}
x(t)\geq \frac{x(0)\exp\left(-c_1\int_0^t y(s)ds\right)}{1+b_1x(0)\int_0^t \exp\left(-c_1\int_0^s y(u)du\right)}>0.\nonumber\end{equation}
In a similar way, we can treat the intervals $[\tau,2\tau],\ldots,[n\tau,(n+1)\tau],n\in N$. Thus by induction, we establish that $x(t)> 0$, $y(t)> 0$ and $z_2(t)> 0$ for all $t\geq 0$.
\par Finally from third equation of (\ref{eq2.1}) and (\ref{eq2.4}),  we have
\[ z_1(t)=\int_{t-\tau}^t\alpha_2y(s)z_2(s)e^{-d_2(t-s)}ds.\]
Therefore the positivity of $z_1(t)$ follows.
\end{proof}
\begin{lem}\label{boundedness}
The solutions of the system (\ref{eq2.1}) with initial conditions (\ref{eq2.3}) and (\ref{eq2.4}) are bounded.
\end{lem}
\begin{proof}
Let $V(t)=\alpha_1\alpha_2x(t)+c_1\alpha_2y(t)+c_1c_2z_1(t)+c_1c_2z_2(t)$, calculating the derivative of $V(t)$ with respect to $t$ along the positive solution of the system (\ref{eq2.1}), we have
\[ \dot{V}(t)=a_1\alpha_1\alpha_2x(t)-b_1\alpha_1\alpha_2x^2(t)-c_1d_1\alpha_2y(t)-c_1c_2d_2z_1(t)-c_1c_3d_3z_2(t). \]
Taking $p=\min\{a_1,d_1,d_2,d_3\}$, we obtain that
\[ \dot{V}(t)+pV(t)\leq 2a_1\alpha_1\alpha_2x(t)-b_1\alpha_1\alpha_2x^2(t).\]
Hence there exists a positive constant $K=a_1^2\alpha_1\alpha_2/b_1$ such that
\[ \dot{V}(t)+pV(t)\leq K,\]
thus, we get
\[ V(t)\leq \left( V(0)-K/p\right)e^{-pt}.\]
Therefore, $V(t)$ is bounded, i.e., each solution of the system (\ref{eq2.1}) is bounded.
\end{proof}
In the next section, we will investigate the feasible equilibrium of the system (\ref{eq2.1}) and study their stability.
\section{Nonnegative equilibria and their local stability}\label{stability}
 In this section, our main objective is to investigate the local behavior of all feasible equilibria and existence of a Hopf bifurcation at  interior equilibrium. The equation for the variable $z_1$ in the system (\ref{eq2.1}) can be rewritten as
\begin{eqnarray}
\frac{dz_1(t)}{dt}&=&\alpha_2y(t)z_2(t)-d_2z_1(t)-\alpha_2e^{-d_2\tau}y(t-\tau)z_2(t-\tau)\nonumber\\
&:=&-d_2z_1(t)+f\left( y(t),z_2(t),y(t-\tau),z_2(t-\tau)\right),\nonumber
\end{eqnarray}
if $y(t)$, $z_2(t)$ are bounded and $y(t)\rightarrow y^*$, $z_2\rightarrow z_2^*$ as $t\rightarrow \infty$, then $z_1(t)\rightarrow f\left( y^*,z_2^*,y^*,z_2^*\right)/d_2$ as $t\rightarrow \infty$, i.e., the asymptotic behavior of $z_1$ is completely dependent on $y(t)$ and $z_2(t)$. Hence, the asymptotic behavior of our proposed model will remain the same with the following reduced system:
\[ \frac{dx(t)}{dt}=x(t)\left(a_1-b_1x(t)-c_1y(t) \right),\]
\begin{equation}
\frac{dy(t)}{dt}=y(t)\left(\alpha_1x(t)-d_1-c_2z_2(t)\right),
\label{eq3.1}
\end{equation}
\[ \frac{dz_2(t)}{dt}=\alpha_2e^{-d_2\tau}y(t-\tau)z_2(t-\tau)-d_3z_2(t),\]
Using simple algebraic manipulations, we get four feasible equilibria for the system (\ref{eq3.1}), namely,
\begin{enumerate}
\item[(a)] trivial equilibrium $E_0(0,0,0)$;
\item[(b)] boundary equilibrium $E_1\left( a_1/b_1,0,0\right)$;
\item[(c)] planner equilibrium $E_2(\bar{x},\bar{y},0)$ exists only when ({\bf H1}) $a_1\alpha_1>b_1d_1$;
\item[(d)] interior equilibrium $E_3(x^*,y^*,z_2^*)$  exists if ({\bf H2}) $a_1\alpha_1\alpha_2>\Delta$.
\end{enumerate}
Where \[\bar{x}=\frac{d_1}{\alpha_1},\ \bar{y}=\frac{a_1\alpha_1-b_1d_1}{c_1\alpha_1},\ x^*=\frac{a_1\alpha_2-c_1d_3e^{d_2\tau}}{b_1\alpha_2},\  y^*=\frac{d_3e^{d_2\tau}}{\alpha_2},\]
\[  z_2^*=\frac{a_1\alpha_1\alpha_2-\Delta}{b_1c_2\alpha_2}, \ \Delta=c_1d_3\alpha_1e^{d_2\tau}+b_1d_1\alpha_2.\]
Further, ({\bf H2}) implies that
\[\tau<\frac{1}{d_2}\log\left( \frac{a_1\alpha_1\alpha_2-b_1d_1\alpha_2}{c_1d_3\alpha_1}\right):=\bar{\tau}.\]
The characteristic equation for trivial equilibrium $E_0(0,0,0)$ is given by
\begin{equation}
(\lambda-a_1)(\lambda+d_1)(\lambda-d_3)=0.\label{chartrivial}
\end{equation}
The characteristic equation (\ref{chartrivial}) has one positive and two negative roots, hence equilibrium $E_0$ is a unstable saddle point.
\par Similarly, the characteristic equation for boundary equilibrium $E_1$ is as follows:
\begin{equation}
(\lambda+d_3)(\lambda+a_1)\left(\lambda+d_1-\frac{a_1\alpha_1}{b_1}\right)=0.\label{charone}
\end{equation}
 Clearly, all the eigenvalues are negative only when $a_1\alpha_1<b_1d_1$, which stabilize $E_1$, otherwise it is unstable.
 Again, the characteristic equation for planner equilibrium $E_2(\bar{x},\bar{y},0)$ becomes
\begin{equation}
\left(\lambda+d_3-\alpha_2\bar{y}e^{-d_2\tau}\right)\left(\lambda^2+b_1\bar{x}\lambda+c_1\alpha_1\bar{x}\bar{y} \right)=0.\label{charboundary}
\end{equation}
Since, both roots of the quadratic equation $\lambda^2+b_1\bar{x}\lambda+c_1\alpha_1\bar{x}\bar{y}=0$ have negative real parts, hence the equilibrium $E_2$ is locally asymptotically stable if $\alpha_2\bar{y}e^{-d_2\tau}<d_3$ for all $\tau>(1/d_2)\log \alpha_2 \bar{y}/d_3:=\tau_{cr}$.
\par Finally, the characteristic equation for interior equilibrium $E_3(x^*,y^*,z_2^*)$ is given as:
\begin{equation}
\lambda^3+A_1(\tau)\lambda^2+A_2(\tau)\lambda+A_3(\tau)+\left(B_1(\tau)\lambda^2+B_2(\tau)\lambda+B_3 \right)e^{-\lambda \tau}=0,
\label{charinterior}\end{equation}
where
\[A_1(\tau)=b_1x^*(\tau)+d_3,\ A_2(\tau)=b_1d_3x^*(\tau)+c_1\alpha_1x^*(\tau)y^*(\tau),\]
\[A_3(\tau)=c_1d_3\alpha_1x^*(\tau)y^*(\tau),\  B_1=-\alpha_2y^*(\tau)e^{-d_2\tau},\]
\[B_2(\tau)=\alpha_2y^*(\tau)\left( c_2z_2^*(\tau)-b_1x^*(\tau)\right)e^{-d_2\tau},\] \[B_3(\tau)=\alpha_2x^*(\tau)y^*(\tau)\left(b_1c_2z_2^*(\tau)-c_1\alpha_1y^*(\tau)\right)e^{-d_2\tau}.\]
We write $A_i$, $B_i$ in place of $A_i(\tau)$, $B_i(\tau)$ for $i=1,2,3$ in the rest of the analysis.\\
The characteristic equation (\ref{charinterior}) can be rewritten as:
\begin{equation}
P(\lambda,\tau)+Q(\lambda,\tau)e^{-\lambda\tau}=0,
\label{eqPQ}\end{equation}
where
\begin{equation}
P(\lambda,\tau)=\lambda^3+A_1\lambda^2+A_2\lambda+A_3,\ Q(\lambda,\tau)=B_1\lambda^2+B_2\lambda+B_3.
\label{eq3.2}\end{equation}
When $\tau=0$, the characteristic equation (\ref{charinterior}) becomes
\begin{equation}
\lambda^3+b_1x^*\lambda^2+(c_1\alpha_1x^*y^*+c_2\alpha_2y^*z_2^*)\lambda+b_1c_2\alpha_2x^*y^*z_2^*=0.
\label{charinteriortau0}\end{equation}
Since $(A_1(0)+B_1(0))(A_2(0)+B_2(0))-(A_3(0)+B_3(0))=b_1c_1\alpha_1{x^*}^2y^*>0$, therefore, using Routh-Hurwitz criterion, all the solutions of the characteristic equation (\ref{charinteriortau0}) have negative real parts. Thus the interior equilibrium $E_3$ is locally asymptotically stable for $\tau=0$ if it exists.
\par In the following, we investigate the existence of purely imaginary roots $\lambda=i\omega(\omega>0)$ of characteristic equation (\ref{eqPQ}). We apply Beretta and Kuang \citep{BerettaKuang2002} geometric criterion which gives the existence of purely imaginary roots of a characteristic equation with delay dependent coefficients.
\begin{lem}\label{pureimaginary}
If ({\bf H2}) holds, then the following are satisfies:
\begin{enumerate}
\item $P(0,\tau)+Q(0,\tau)\neq 0$,
\item $P(i\omega,\tau)+Q(i\omega,\tau)\neq 0$ for all $\omega\in R$,
\item $\lim\sup\{|Q(\lambda,\tau)/P(\lambda,\tau)|\ : \ |\lambda|\rightarrow \infty,\ Re\lambda \geq 0 \}<1$,
\item $F(\omega,\tau)=|P(i\omega,\tau)|^2-|Q(i\omega,\tau)|^2$ for each $\tau$ has at most a finite number of real zeros,
\item each positive root $\omega(\tau)$ of $F(\omega,\tau)=0$ is continuous and differentiable in $\tau$ whenever it exists.
\end{enumerate}
\end{lem}
\begin{proof}
1. For $\tau\in[0,\bar{\tau})$,
\[ P(0,\tau)+Q(0,\tau)=b_1c_2\alpha_2x^*y^*z_2^*e^{-d_2\tau}\neq 0. \]
2. $P(i\omega,\tau)+Q(i\omega,\tau)=-(A_1+B_1)\omega^2+A_3+B_3+i[-\omega^3+(A_2+B_2)\omega]\neq 0$.\\
3. Since $P(\lambda,\tau)$ is a third degree polynomial in $\lambda$ and $Q(\lambda,\tau)$ second degree, hence,  $\lim\ \sup\{|Q(\lambda,\tau)/P(\lambda,\tau)|\ : \ |\lambda|\rightarrow \infty,\ Re\lambda \geq 0 \}=0<1 $.\\
4. Let $F$ be defined as\\
 \[F(\omega,\tau)=|P(i\omega,\tau)|^2-|Q(i\omega,\tau)|^2.\]
 From
 \[|P(i\omega,\tau)|^2=\omega^6 +\left(A_1^2-2A_2\right)\omega^4+\left(A_2^2-2A_1A_3
  \right)\omega^2+A_3^2 \]
and
\[|Q(i\omega,\tau)|^2=B_1^2\omega^4 +\left(B_2^2-2B_1B_3\right)\omega^2+B_3^2, \]
we have
\[F(\omega,\tau)=\omega^6+p(\tau)\omega^4+q(\tau)\omega^2+r(\tau), \]
where
\[ p(\tau)=\frac{(a_1\alpha_2-c_1d_3e^{d_2\tau})(a_1b_1\alpha_2-(b_1+2\alpha_1)c_1d_3e^{d_2\tau})}{b_1\alpha_2^2},\]
\[ q(\tau) = c_1^2\alpha_1^2{x^*}^2{y^*}^2+2b_1c_2\alpha_2x^*y^*z_2^*e^{-d_2\tau}-c_2\alpha_2^2{y^*}^2z_2^*e^{-2d_2\tau(2b_1x^*+c_2z_2^*)},\]
\[ r(\tau)=-3c_1d_3\alpha_1+(a_1\alpha_1-b_1d_1)\alpha_2e^{-d_2\tau}.\]
It is obvious that property (iv) is satisfied.\\
5. Since $F(\omega,\tau)$ is continuous in $\omega$ and $\tau$ and it is differentiable with respect to $\omega$. Therefore, from Implicit Function Theorem each root of $F(\omega,\tau)=0$ is continuous and differentiable in $\tau$.
\par Hence all the conditions of the Lemma are satisfied, which ensure the existence of purely imaginary roots for the characteristic equation (\ref{charinterior}).
\end{proof}
Now let $\lambda=i\omega\ (\omega>0)$ be a root of (\ref{charinterior}). Substituting it into (\ref{charinterior}) and separating the real and imaginary parts, we get
\[ B_2\omega \sin\omega\tau+\left(B_3-B_1\omega^2\right)\cos\omega\tau=A_1\omega^2-A_3,\]
\begin{equation}
\left(B_1\omega^2-B_3 \right)\sin\omega\tau + B_2\omega \cos\omega\tau=\omega^3-A_2\omega,
\label{sepri}
\end{equation}
which gives
\[ \sin\omega\tau=\frac{B_1\omega^5-\left(A_1B_2+A_2B_1+B_3\right)\omega^3+\left(A_2B_3+A_3
B_2\right)\omega}{B_1^2\omega^4+(B_2^2-2B_1B_3)\omega^2+B_3^2},\]
 \begin{equation}
 \cos\omega\tau=\frac{\left(B_2-A_1B_1\right)\omega^4+\left(A_1B_3+A_3B_1-A_2B_2\right)\omega^2-A_3B_3}{B_1^2\omega^4+(B_2^2-2B_1B_3)\omega^2+B_3^2}.
\label{sincosomega} \end{equation}
We can define the angle $\theta(\tau)\in[0,2\pi]$, $\forall$ $\tau\geq 0$ as the solution of (\ref{sincosomega}):
\[\sin\theta(\tau)=\frac{B_1\omega^5-\left[A_1B_2+A_2B_1+B_3\right]\omega^3+\left[A_2B_3+A_3B_2\right]\omega}{B_1^2\omega^4+[B_2^2-2B_1B_3]\omega^2+B_3^2}, \]
 \begin{equation}
 \cos\theta(\tau)=\frac{\left[B_2-A_1B_1\right]\omega^4+\left[A_1B_3+A_3B_1-A_2B_2
 \right]\omega^2-A_3B_3}{B_1^2\omega^4+[B_2^2-2B_1B_3]\omega^2+B_3^2},
\label{sincostheta} \end{equation}
where $\omega=\omega(\tau)$ and such $\theta(\tau)$ is uniquely well defined for  all $\tau$, so that $F(\omega(\tau),\tau)=0$. Hence  \begin{equation}
  \omega^6+p(\tau)\omega^4+q(\tau)\omega^2+r(\tau)=0.
 \label{omegaequation} \end{equation}
Again the polynomial function $F$ can be written as
 \[ F(\omega,\tau)=h(\omega^2,\tau),\]
 where $h$ is a cubic polynomial, defined by
 \[ h(z,\tau):=z^3+p(\tau)z^2+q(\tau)z+r(\tau).\]
Applying the Descartes' rule of signs for the number of positive roots of $h(z,\tau)=0$, we get the following four cases: \\
{\bf Case I}: Let
\[I_{11}=\{\tau\geq 0\ |\ p(\tau)>0\,\ q(\tau)>0\  \mbox{and} \ r(\tau)>0\}, \]
\[I_{12}=\{\tau\geq 0\ | \ r(\tau)>0,\ \mbox{at least one } p(\tau)<0\, \ q(\tau)<0\},\]
\[ I_1=I_{11}\cup I_{12}.\]
In the interval $I_{12}$, $h(z,\tau)=0$ either has $0$ or $2$ positive roots.  When the polynomial $h(z,\tau)$ has no positive zero in $I_{12}$, then it also has no positive zero in the interval $I_1$. Thus in this case, purely imaginary root of (\ref{charinterior}) never exists.\\
{\bf Case II}: Let
\[I_{21}=\{\tau\geq 0\ |\ p(\tau)>0\ \mbox{and} \ r(\tau)<0\},\]
\[I_{22}=\{\tau\geq 0\ | \ p(\tau)<0,\ q(\tau)<0\  \mbox{and} \ r(\tau)<0\},\]
\[I_{23}=\{\tau\geq 0\ | \ P(\tau)<0,\ q(\tau)>0\ \mbox{and}\ r(\tau)<0\}.\]
In the region $I_{23}$, either one or three positive zeros of $h(z,\tau)$ exist. Suppose only one positive zero is feasible in $I_{23}$, then in the interval $I_2=I_{21}\cup I_{22}\cup I_{23}$, $h(z,\tau)$ has only one positive zero. Therefore, $i\omega^*$ with $\omega^*=\omega(\tau^*)>0$ is a purely imaginary root of (\ref{charinterior}) if and only if $\tau^*$ is a zero of  the $S_n$, where
\[S_n(\tau)=\tau-\frac{\theta(\tau)+2n\pi}{\omega(\tau)},\ \tau\in I_2,\ \mbox{with}\ n\in N_0.\]
Now we will verify the following lemma:
\begin{lem}[Beretta and Kuang \citep{BerettaKuang2002}]\label{BerettaKuang}
Assume that $\omega(\tau)$ is a positive real root of (\ref{charinterior}) defined for $\tau\in I$, $I\subseteq R_{+0}$, and at some $\tau^*\in I$,
  \begin{equation}
  S_n(\tau^*)=0, \ \ \mbox{for some}\ n\in N_0.
  \end{equation}
Then a pair of simple conjugate pure imaginary roots $\lambda_+(\tau^*)=i\omega(\tau^*)$, $\lambda_-(\tau^*)=-i\omega(\tau^*)$ of (\ref{charinterior}) exists at  $\tau=\tau^*$ which crosses the imaginary axis from left to right if $\delta(\tau^*)>0$ and crosses the imaginary axis from right to left if $\delta(\tau^*)<0$, where
{\small
\begin{equation}
\delta(\tau^*)=Sign\left\{ \left.\frac{dRe(\lambda)}{d\tau}\right|_{\lambda=i\omega(\tau^*)}\right\}=Sign\left\{ \frac{\partial F}{\partial \omega}(\omega(\tau^*),\tau^*)\right\}Sign\left\{ \left.\frac{dS_n(\tau)}{d\tau}\right|_{\tau=\tau^*}\right\}.
\label{derivative}\end{equation}
}
\end{lem}
Since $\left.\partial F(\omega,\tau)/\partial \omega\right|_{\omega=\omega(\tau^*)}=\left[6\omega^5+4p(\tau)\omega^3+2q(\tau)\omega\right]_{\omega=\omega(\tau^*)}> 0$, therefore, from (\ref{derivative}), we get
\[\delta(\tau^*)=Sign\left\{ \left.\frac{dRe(\lambda)}{d\tau}\right|_{\lambda=i\omega(\tau^*)}\right\}=Sign\left\{\left.\frac{dS_n(\tau)}{d\tau}\right|_{\tau=\tau^*}\right\}.
\]
Here, we can easily observe that $S_n(0)<0$ and $S_n(\tau)>S_{n+1}(\tau)$ $\forall$ $\tau\in I_2$, $n\in N_0$. Thus, if  $S_0$ has no zero in $I_2$, then the function $S_n$ also have no zero in $I_2$ and if the function $S_n$ has positive zeros, denoted by $\tau_n^j$ for some $\tau\in I_2,\ n\in N_0$, then without loss of generality, we may assume that
 \[\frac{dS_n(\tau_n^j)}{d\tau}\neq 0\ \mbox{with}\ S_n(\tau_n^j)=0.\]
 Applying similar logic as in \citep{BerettaKuang2002}, stability switches occur at the zeros of $S_0(\tau)$, denoted by $\tau_0^j$, if ({\bf H2}) holds. Let us assume that
 \[\tau^*=\min\{\tau\in I_2\ | \ S_0(\tau)=0\}\ \ \mbox{and}\ \ \tau^{**}=\max\{\tau\in I_2\ | \ S_0(\tau)=0\}. \]
 Using the Hopf bifurcation theorem for functional differential equation \citep{HaleLunel1993}, we can conclude the existence of Hopf bifurcation in the following theorem:
 \begin{thm}\label{theoremforbifurcation}
Let ({\bf H2}) hold. The local behavior of the system (\ref{eq2.1}) at interior equilibrium is described as:
\begin{enumerate}
\item If the function $S_0(\tau)$ has no positive zero in $I_2$, then the interior equilibrium $E_3(x^*,y^*,z_2^*)$ is locally asymptotically stable for all $\tau\geq 0$.
\item If the function $S_n(\tau)$ has at least positive zero in $I_2$ for some $n\in N_0$, then $E_3$ is locally asymptotically stable for $\tau\in[0,\tau^*)\cup(\tau^{**},\bar{\tau}]$ and unstable and a Hopf bifurcation occurs for $\tau\in(\tau^*,\tau^{**})$, i.e., stability switches of stability-instability-stability occur.\end{enumerate}
  \end{thm}
  \noindent
{\bf Case III}: If in the interval $I_3=I_{12}$, $h(z,\tau)=0$ has two positive roots, denoted by $\omega_1$ and $\omega_2$, we get following two sequences of functions on $I_3$:
\[S_n^{(1)}(\tau)=\tau-\frac{\theta_1(\tau)+2n\pi}{\omega_1(\tau)} \ ;\ S_n^{(2)}(\tau)=\tau-\frac{\theta_2(\tau)+2n\pi}{\omega_2(\tau)},\  \ n\in N_0,\]
where $\theta_1(\tau)$ and $\theta_2(\tau)$ are the solutions of (\ref{sincostheta}) when $\omega=\omega_1,\omega_2$ respectively. Similarly, we can also obtain for $\tau\in I_3$ that $S_n^{(k)}(0)\leq 0$ and $S_n^{(k)}(\tau)>S_{n+1}^{(k)}(\tau)$ with $n\in N_0$, $k=1,2$. Thus if $S_0^{(1)}(\tau)>S_0^{(2)}(\tau)$, then $S_n^{(1)}(\tau)>S_n^{(2)}(\tau)$ and hence stability switch depends on all real roots of $S_n^{(1)}(\tau)=0$, otherwise the stability switches depend on roots of both $S_n^{(1)}(\tau)=0$ and $S_n^{(2)}(\tau)=0$. Furthermore, we can also obtain the similar results stated in Theorem \ref{theoremforbifurcation}.\\
{\bf Case IV}: If in the interval $I_4=I_{23}$, $h(z,\tau)$ has three positive zeros, we can obtain the parallel results as in case III.
\section{Direction and stability of Hopf bifurcation}\label{direction}
 In the previous section, we obtained the conditions, under which system (\ref{eq3.1}) undergoes Hopf
bifurcation, taking maturation delay ($\tau$) as  the critical parameter. Using the normal form theory and center manifold reduction as described in Hassard et al. \citep{HassardKazarinoff1981}, we can investigate the direction of Hopf bifurcation and the properties of these bifurcating periodic solutions. Hence, we always assume that system (\ref{eq3.1}) undergoes Hopf bifurcations at the critical value $\tau^*$ of $\tau$ and there exists a pair of pure imaginary roots, i.e., $\pm i\omega(\tau^*)$ of the characteristic equation (\ref{charinterior}).
\par Using the Appendix \ref{appA}, we can compute the following values:
\[C_1(0)=\frac{i}{2\omega^*\tau^*}\left(g_{20}g_{11}-2|g_{11}|^2-\frac{|g_{02}|^2}{3} \right)+\frac{g_{21}}{2},\]
\[\mu_2=\frac{Re\{C_1(0)\}}{Re\{\lambda'(\tau^*)\}},\]
\begin{equation}
\beta_2=2Re\{C_1(0)\},
\label{parameters}\end{equation}
\[ T_2=\frac{I_m\{C_1(0)\}+\mu_2I_m\{ \lambda'(\tau^*)\}}{\omega^*\tau^*},\]
which determine the behavior of bifurcating periodic solution in the center manifold at the critical value $\tau^*$, i.e., $\mu_2$
determines the direction of the Hopf bifurcation: if $\mu_2 > 0 \ (\mu_2 < 0)$, then the Hopf bifurcation is supercritical
(subcritical) and the bifurcating periodic solution exists for $\tau > \tau^*\  (\tau < \tau^*)$; $\beta_2$ determines the stability of the
bifurcating periodic solution: the bifurcating periodic solution is stable (unstable) if $\beta_2 < 0 \ (\beta_2 > 0)$ and $T_2$
determines the period of the bifurcating periodic solution: the period increases (decreases) if $T_2 > 0 \ (T_2<0)$.
\section{Numerical simulation}\label{numerical}
To verify the previously established results, consider a three species crop-pest-natural enemy stage structured food chain model with the following parameter values:
\[ \frac{dx(t)}{dt}=x(t)\left(a_1-x(t)-y(t) \right),\]
\begin{equation}
\frac{dy(t)}{dt}=y(t)\left(\alpha_1 x(t)-d_1-0.6z_2(t)\right),
\label{numeq}
\end{equation}
\[\frac{dz_2(t)}{dt}=1.3e^{-0.4\tau}y(t-\tau)z_2(t-\tau)-0.3z_2(t).\]
If we choose $a_1=1$, $\alpha_1=0.1$ and $d_1=0.5$, then we obtain that for the trivial equilibrium $E_0(0,0,0)$, the characteristic equation has three eigenvalues, $\lambda_1=1,\ \lambda_2=-0.5$ and $\lambda_3=0.3$. This verifies that trivial equilibrium of the system (\ref{numeq}) is a unstable saddle point. Similarly, the characteristic equation for the boundary equilibrium $E_1(1,0,0)$ has $\lambda_1=-1, \lambda_2=-0.3$ and $\lambda_3=-0.2$ eigenvalues and hence the boundary equilibrium is locally asymptotically stable (see Figure \ref{nonnegativestable}(a)).
\par Now let $a_1=2$, $d_1=0.05$ and $\alpha_1=1.2$. Then the condition ({\bf H1}) for the positivity of equilibrium $E_2$ is satisfy and the characteristic equation for the equilibrium $E_2(0.0416667,1.95833,0)$ of the system (\ref{numeq}) has three eigenvalues, namely, $\lambda_1=-0.0208333 - 0.312222 i$, $\lambda_2=-0.0208333 + 0.312222 i$ and $\lambda_3=-0.3 + 2.54583 e^{-0.4 \tau}$ eigenvalues. Here, $\lambda_3$ is negative only when $\tau>\bar{\tau}=5.34608$, in this case, the planner equilibrium $E_2$ is locally asymptotically stable for $\tau>5.34608$ (see Figure \ref{nonnegativestable}(b)). We, mainly focus on the dynamics of the interior equilibrium. To study the local behavior of interior equilibrium $E_3(x^*(\tau),y^*(\tau),z_2^*(\tau))$ of the system (\ref{numeq}), we take the same parameters values for $a_1,\ d_1$ and $\alpha_1$. We obtain that the equilibrium is positive if $\tau<\tau_1=5.34608$ ( see Figure \ref{positiveequilibrium}).
\begin{figure}[!]
\centering
\includegraphics[height=2in,width=5.5in]{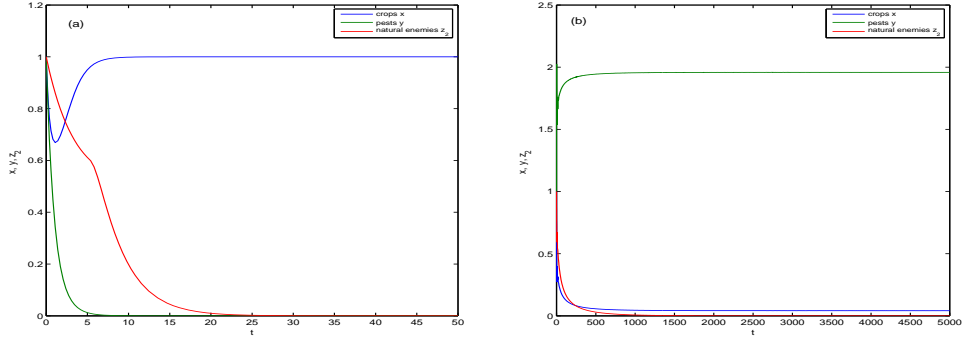}\\
\caption{Nonnegative equilibrium points are stable (a) boundary equilibrium $E_1$ is stable for $a_1=1,\ d_1=0.5,\ \alpha_1=0.1$; (b) planner equilibrium $E_2$ is stable for $a_1=2,\ d_1=0.05,\ \alpha_1=1.2$ and $\tau=5.4>\tau_{cr}=5.34608$.}
\label{nonnegativestable}
\end{figure}

\begin{figure}[!]
\centering
\includegraphics[height=2in,width=5.5in]{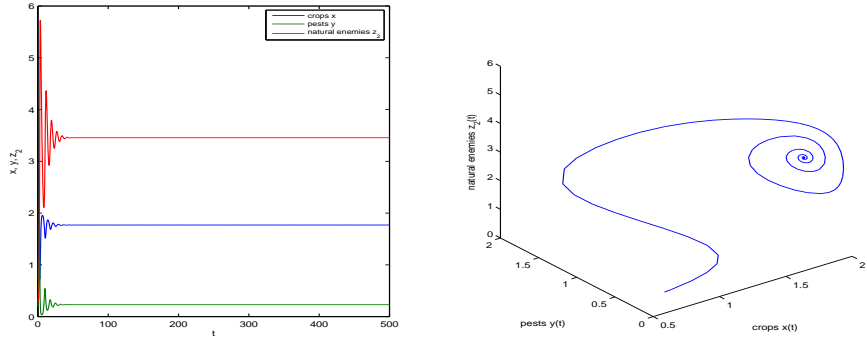}\\
\caption{The interior equilibrium $E_3$ is locally asymptotically stable for $\tau=0$.}
\label{stablewithoutstage}
\end{figure}

\begin{figure}[!]
\centering
\includegraphics[height=3in,width=4in]{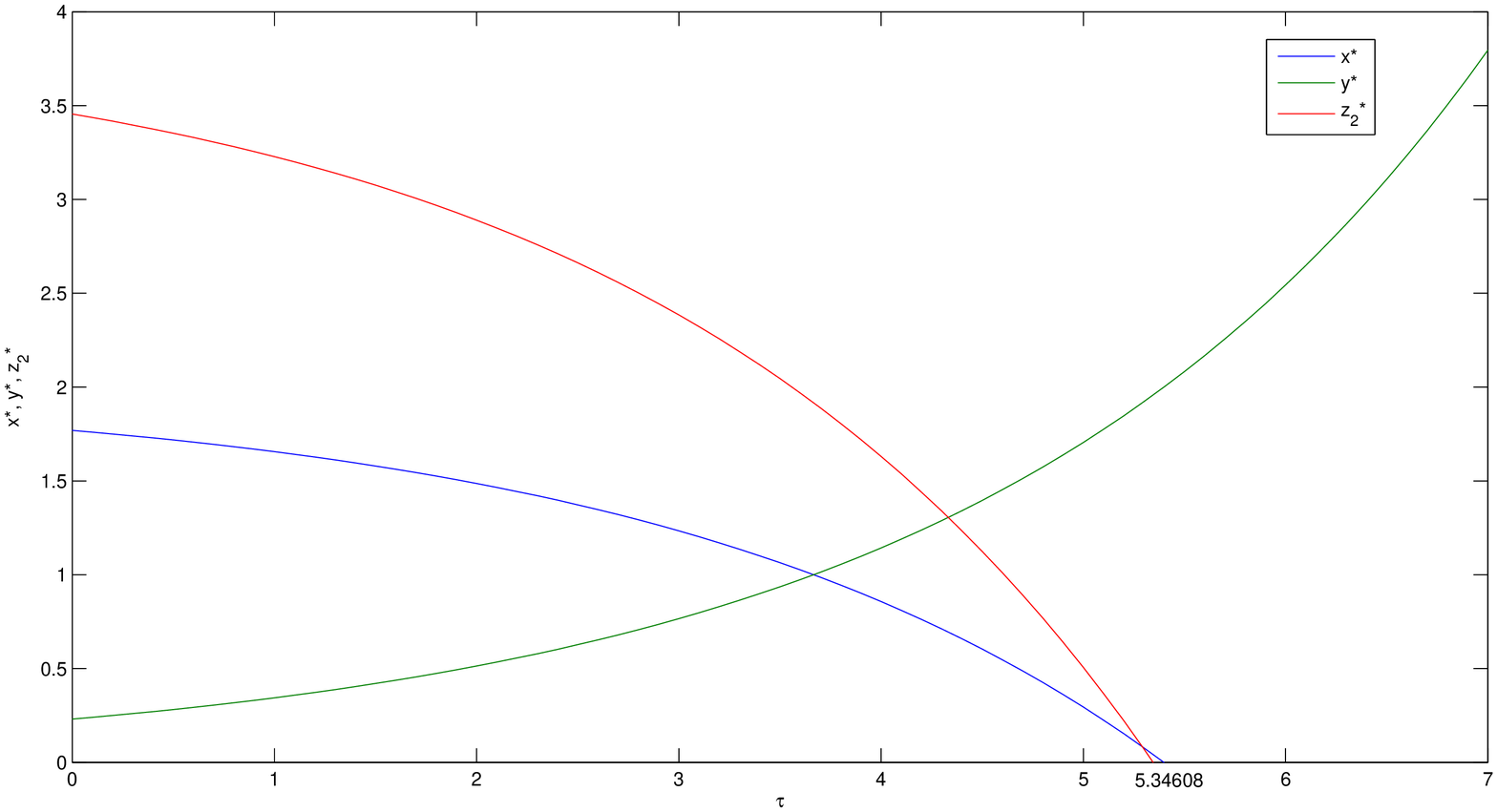}\\
\caption{The equilibrium $E_3(x^*,y^*,z_2^*)$ is positive for $\tau<\bar{\tau}=5.34608$.}
\label{positiveequilibrium}
\end{figure}
\par To apply the Descartes' rule of signs, we plot the coefficients $p(\tau)$, $q(\tau)$ and $r(\tau$) of the function $h(z,\tau)$ with the maturation delay $\tau$ (see Figure \ref{omegapositive}). We can easily check that in the intervals $I_{12}, I_{23}$, the function $h(z,\tau)=0$ has $0$ and $1$ positive roots respectively. Thus exactly one zero of $h(z,\tau)$ is feasible in only the interval $I_2=[0,2.59955)$ (see Figure \ref{omegapositive}(d)).
  Further, all the conditions of the lemma \ref{pureimaginary} are satisfied in the interval $I_2$. Now taking $S_0$ on one axis and $\tau$ on another axis, we obtain that the function $S_0(\tau)=\tau-\theta(\tau)/\omega(\tau)$ has two zeros $\tau^*\approx 0.743$ and $\tau^{**}\approx 1.568$ in the interval $I_2$, i.e., there are two critical values of the maturation delay of natural enemies at which the stability switching occurs (see Figure \ref{criticalvalue}).

 \begin{figure}[!]
\centering
\includegraphics[height=4in,width=5.5in]{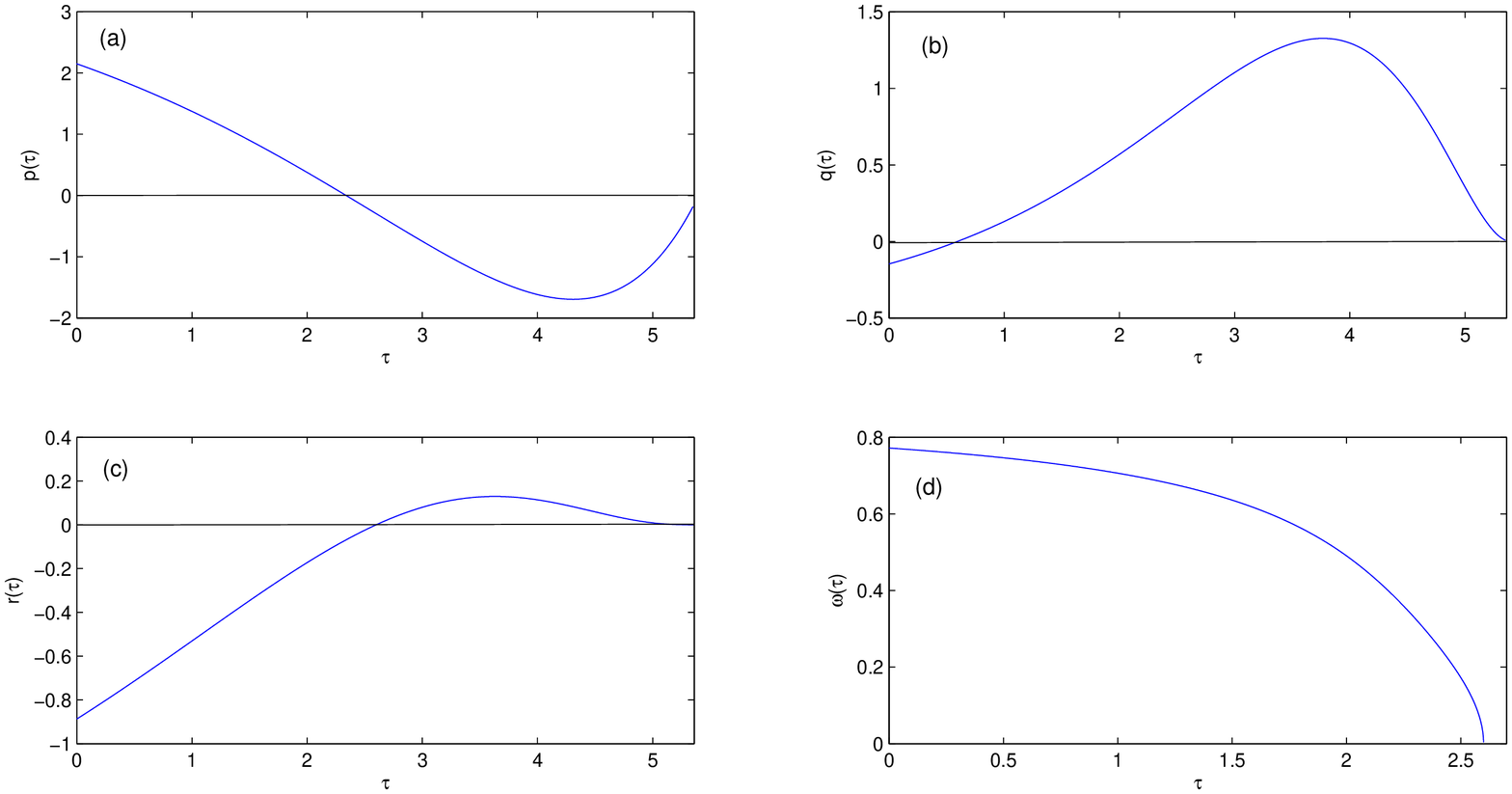}\\
\caption{(a) $p(\tau)$;  (b) $q(\tau)$  and (c) $r(\tau)$ in the interval $[0,5.34608)$; (d) curve of $\omega(\tau)$ on $\tau\in[0,2.59955)$.}
\label{omegapositive}
\end{figure}

\begin{figure}[!]
\centering
\includegraphics[height=2in,width=5.5in]{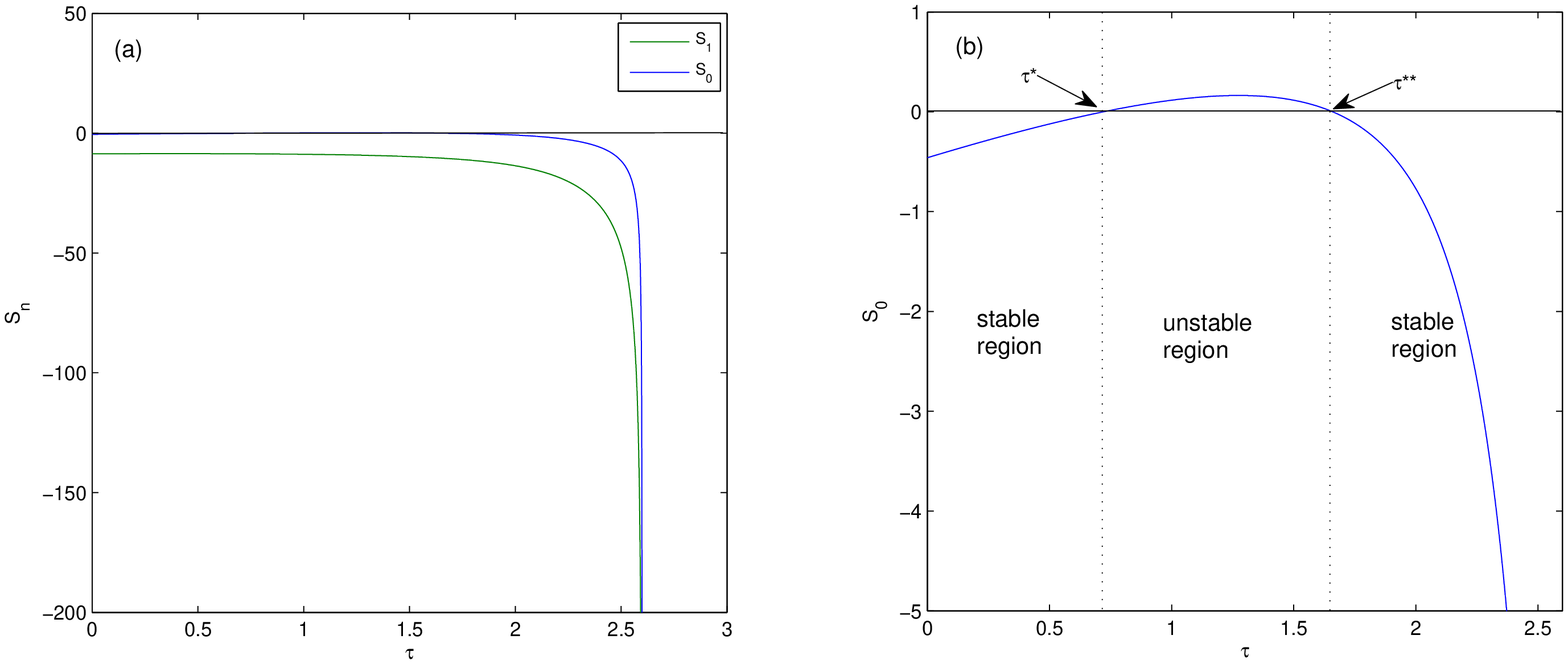}\\
\caption{(a) Graphs of functions $S_0$ and $S_1$ for $\tau\in[0,2.59955)$; (b) curve of $S_0$ for the same value of $\tau$.}
\label{criticalvalue}
\end{figure}
\par Using the same set of parametric values, it is obtained that the interior equilibrium $E_3(x^*,y^*,z_2^*)$ is locally asymptotically stable if $\tau<\tau^*=0.743$ and a Hopf bifurcation occurs if $\tau\geq 0.743$, see Figures \ref{stabilityswitch1}(a) and \ref{stabilityswitch1}(a). The equilibrium again becomes locally asymptotically when $\tau>\tau^{**}=1.568$, see Figure \ref{stabilityswitch2}. Thus the interior equilibrium of the system (\ref{numeq}) is locally asymptotically stable for $\tau\in [0, 0.743)\cup (1.568,5.34608]$ and is unstable for $\tau\in (\tau^*,\tau^{**})$. Hence the stability switches from \emph{stability-instability-stability} occurs. This is the verification of the Theorem \ref{theoremforbifurcation}.
\par Furthermore, for the system (\ref{numeq}), it is clear that $Re(c_1(0))|_{\tau=\tau^*}=-3.9481$ $<0$ and $Re(c_1(0))|_{\tau=\tau^{**}}= 9.3706>0$, according to the formula given in section \ref{direction}. Therefore, Hopf bifurcation for the interior equilibrium at $\tau^*$ (resp. $\tau^{**}$) is forward (resp. backward) and the bifurcating periodic solution on the center manifold are orbitally asymptotically stable (see Figure \ref{bif_per_sol}). Finally, we observe with the following set of parameters: $a_1=7,\  b_1 =1,\  c_1 = 1,\ c_2 =0.5,\ d_1 = 0.05, \ d_2 = 0.6,\ d_3 =1.2,\ \alpha_1 = 1.5,\ \alpha_2 =2$, then  system (\ref{eq2.1}) has a complex dynamics of multiple bifurcation (i.e., chaos) for the maturation delay $\tau=1.5$ (see Figure \ref{chaos}).

\begin{figure}[!]
\centering
\includegraphics[height=2in,width=5.5in]{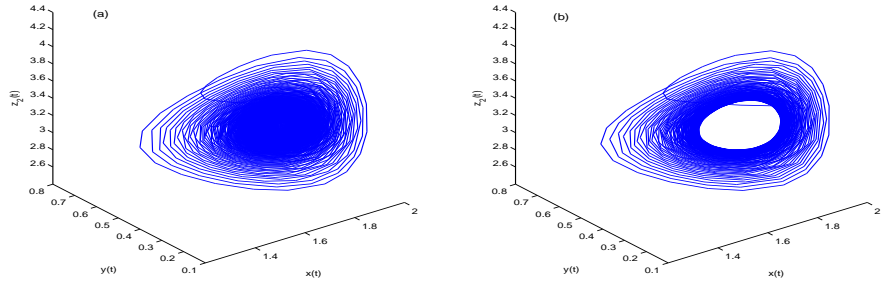}\\
\caption{The dynamics of the interior equilibrium for first critical value of maturation delay $\tau$: (a) $E_3$ is locally asymptotically stable for $\tau=0.742<\tau^*=0.743$; (b) occurrence of Hopf bifurcation at $E_3$ for $\tau=0.75>\tau^*=0.743$.}
\label{stabilityswitch1}
\end{figure}
\begin{figure}[!]
\centering
\includegraphics[height=2in,width=5.5in]{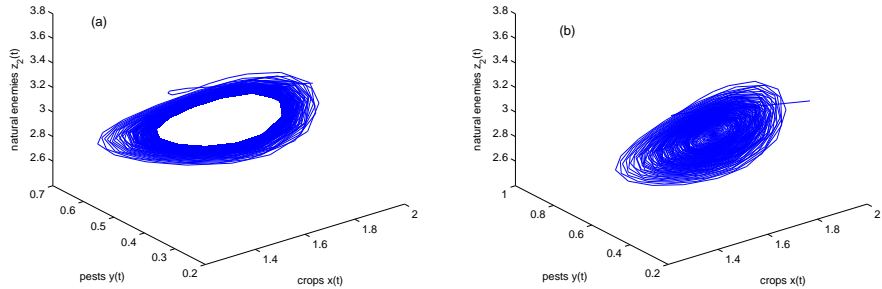}\\
\caption{The dynamics of the interior equilibrium for second critical value of maturation delay $\tau$: (a) occurrence of Hopf bifurcation at $E_3$ for $\tau=1.56<\tau^{**}=1.568$; (b) $E_3$ is locally asymptotically stable for $\tau=1.57>\tau^{**}=1.568$.}
\label{stabilityswitch2}
\end{figure}
\begin{figure}[!]
\centering
\includegraphics[height=2in,width=5.5in]{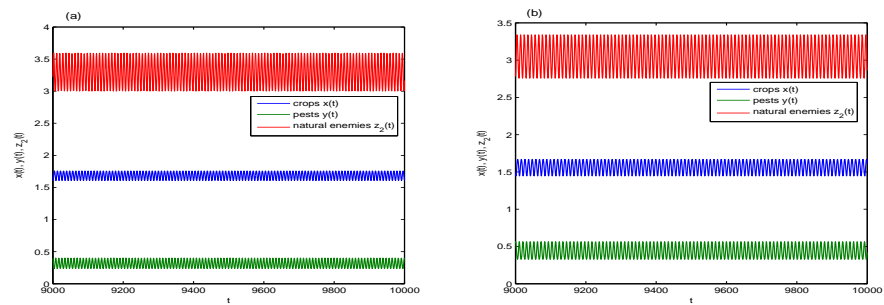}\\
\caption{(a) The bifurcated periodic solution for the interior equilibrium is locally asymptotically stable at $\tau^*=0.743$; (b) The bifurcated periodic solution for the interior equilibrium is locally asymptotically stable at $\tau^{**}=1.568$.}
\label{bif_per_sol}
\end{figure}
\begin{figure}[!]
\centering
\includegraphics[height=2in,width=5.5in]{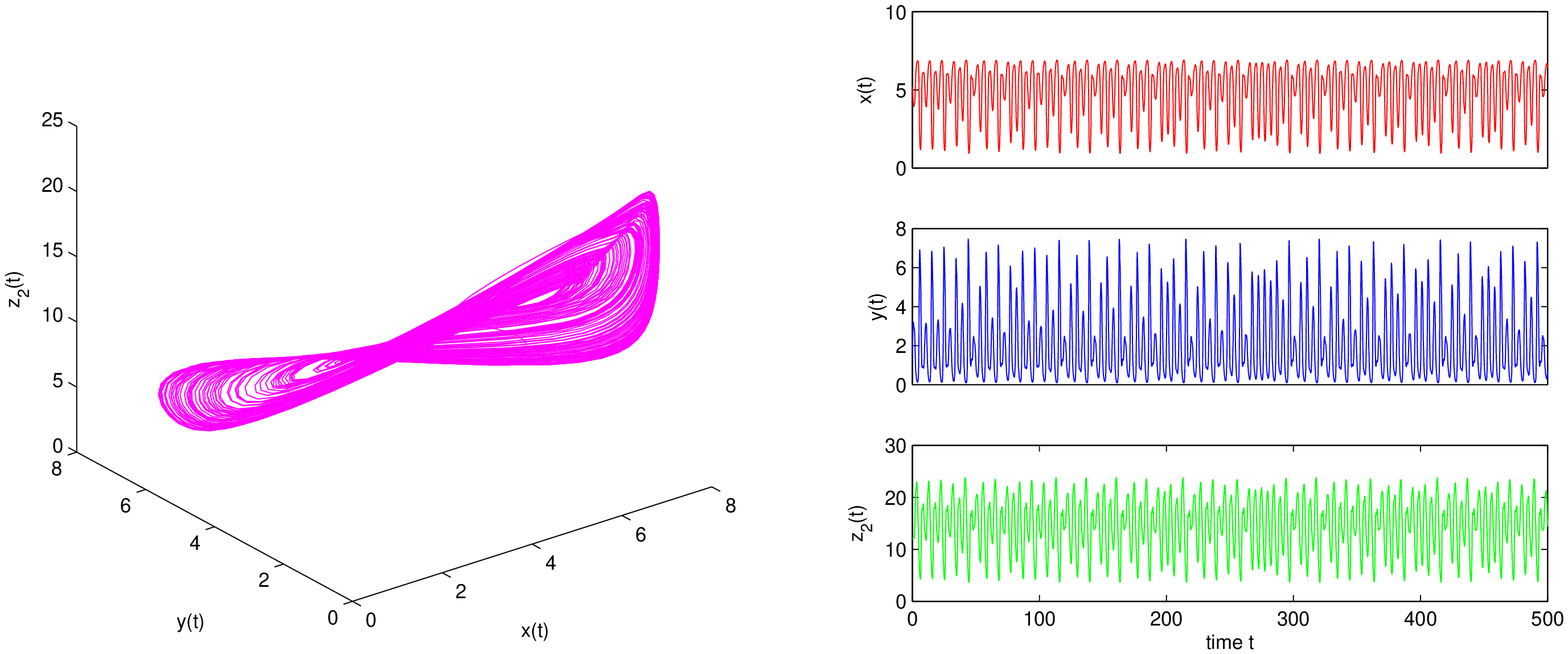}\\
\caption{The chaotic behavior of the system at the interior equilibrium for parameter set: $a_1=7,\  b_1 =1,\  c_1 = 1,\ c_2 =0.5,\ d_1 = 0.05, \ d_2 = 0.6,\ d_3 =1.2,\ \alpha_1 = 1.5,\ \alpha_2 =2$ and $\tau=1.5$.}
\label{chaos}
\end{figure}
\section{Conclusion}
In this paper, we have proposed a three species crop-pest-natural enemy food chain mathematical model with stage structure and maturation delay for the natural enemy. We have studied the local stability of four nonnegative equilibria of the system (\ref{eq2.1}). It is found that the trivial equilibrium $E_0(0,0,0)$ is always unstable; the boundary equilibrium $E_1(a_1/b_1,0,0)$ is locally asymptotically stable if $a_1\alpha_1<b_1d_1$. Further, the planner equilibrium $E_2(\bar{x},\bar{y},0)$ is locally asymptotically stable if $a_1\alpha_1>b_1d_1$ and $\tau>(1/d_2)\log(\alpha_2\bar{y}/d_3)$, otherwise, it is unstable. The interior equilibrium is locally asymptotically stable if $\tau<\tau^*$, a Hopf bifurcation occurs in the interval $\tau^*<\tau<\tau^{**}$ and if the maturation delay crossed the second critical $\tau^{**}$, then the interior equilibrium becomes again stable. Thus the maturation delay plays an important role in switching of stability from \emph{stability-instability-stability}. Furthermore, using a numerical simulation, we obtained that the existence of  bifurcation and also observed the chaotic behavior of the system for a particular range of the maturation delay. In particular, it is observed that the larger maturation delay may lead to extinction of the natural enemy, i.e., natural enemies may extinct due to its stage structure.  This shows that the maturation delay of natural enemy is the controlling parameter for the pest population. Thus stage structure has a great importance in the dynamics  of the three species crop-pest-natural enemy food chain.
\appendix
\section{Appendix}\label{appA}
Let, $x_1=x-x^*$, $x_2=y-y^*$, $x_3=z_2-z_2^*$, $\bar{x}_i(t)=x_i(\tau t)$, $\tau=\tau^*+\mu$ and dropping the bars for simplification
of notations, system (\ref{eq3.1}) is transformed into an FDE in $C=C([-1,0],R^3)$ as
\begin{equation}
\dot{x}(t)=L_{\mu}(x_t)+F(\mu,x_t),
\label{normalize}\end{equation}
where $x(t)=(x_1(t),x_2(t),x_3(t))^T\in R^3$ and $L_{\mu}:C\rightarrow R$, $F:C\times R\rightarrow R$ are given, respectively, by
\begin{eqnarray}
L_{\mu}(\phi)&=&(\tau^*+\mu)
\left(
\begin{array}{ccc}
 -b_1x^* & -c_1x^* & 0 \\
 \alpha_1 y^* & 0 & -c_2y^* \\
 0 & 0 & -d_3 \\
\end{array}
\right)
\left(
  \begin{array}{c}
    \phi_1(0) \\
    \phi_2(0) \\
    \phi_3(0) \\
  \end{array}
\right)\nonumber\\
& & +
(\tau^*+\mu)
\left(
\begin{array}{ccc}
 0 & 0 & 0 \\
 0 & 0 & 0 \\
 0 & \alpha_2z_2^*e^{-d_2(\tau^*+\mu)} & \alpha_2y^*e^{-d_2(\tau^*+\mu)}\nonumber \\
\end{array}
\right)
\left(
  \begin{array}{c}
    \phi_1(-1) \\
    \phi_2(-1) \\
    \phi_3(-1) \\
  \end{array}
\right),\\
\label{Lmu}\end{eqnarray}
and
\begin{equation}
F(\mu,\phi)=(\tau^*+\mu)\left(
                         \begin{array}{c}
                           -b_1\phi_1^2(0)-c_1\phi_1(0)\phi_2(0) \\
                           \alpha_1\phi_1(0)\phi_2(0)-c_2\phi_2(0)\phi_3(0) \\
                           \alpha_2e^{-d_2(\tau^*+\mu)}\phi_2(-1)\phi_3(-1) \\
                         \end{array}
                       \right)
\label{Fmuxt},\end{equation}
where $\phi=(\phi_1,\phi_2,\phi_3)^T\in C$. By the Riesz representation theorem, there exists a function $\eta (\theta,\mu)$ of
bounded variation for $\theta\in[-1,0]$ such that
\begin{equation}
L_{\mu}\phi=\int_{-1}^0\eta(\theta,\mu)\phi(\theta) \ \ \  \mbox{for} \  \theta\in C.
\label{Lmuother}\end{equation}
In fact, we can choose
\begin{eqnarray}
\eta(\theta,\mu)&=&(\tau^*+\mu)
\left(
\begin{array}{ccc}
 -b_1x^* & -c_1x^* & 0 \\
 \alpha_1 y^* & 0 & -c_2y^* \\
 0 & 0 & -d_3 \\
\end{array}
\right)
\delta (\theta)\nonumber\\
& & -
(\tau^*+\mu)
\left(
\begin{array}{ccc}
 0 & 0 & 0 \\
 0 & 0 & 0 \\
 0 & \alpha_2z_2^*e^{-d_2(\tau^*+\mu)} & \alpha_2y^*e^{-d_2(\tau^*+\mu)}\nonumber \\
\end{array}
\right)
\delta(\theta+1),\\
\label{etathetamu}
\end{eqnarray}
where $\delta$ is Dirac delta function. For $\phi\in C^1([-1,0],R^3)$, define
\begin{equation}
A(\mu)\phi=\left\{ \begin{array}{ll}
                     \frac{d\phi(\theta)}{d\theta},&  \theta\in [-1,0), \\[0.1in]
                     \int_{-1}^0 d\eta(s,\mu)\phi(s),& \theta=0,
                   \end{array}\right.
\label{Amuphi}\end{equation}
and
\begin{equation}
R(\mu)\phi=\left\{ \begin{array}{ll}
                     0, & \theta\in [-1,0), \\[0.1in]
                     F(\mu,\phi), & \theta=0.
                   \end{array}\right.
\label{Rmuphi}\end{equation}
Then the system (\ref{normalize}) is equivalent to
\begin{equation}
\dot{x}_t=A(\mu)x_t+R(\mu)x_t,
\label{x_t}\end{equation}
where $x_t(\theta)=x(t+\theta)$ for $\theta\in [-1,0]$.
For $\psi\in C^1([0,1],(R^3)^*)$, define
\begin{equation}
A^*\psi(s)=\left\{ \begin{array}{ll}
                     \frac{-d\psi(s)}{ds}, & s\in (0,1], \\[0.1in]
                     \int_{-1}^0 d\eta^T(s,0)\psi(s), & s=0,
                   \end{array}\right.
\label{A*}\end{equation}
and a bilinear inner product
\begin{equation}
\langle \psi(s),\phi(\theta)\rangle =\bar{\psi}(0)\phi(0)-\int_{-1}^0\int_{\xi=0}^{\theta}\bar{\psi}(\xi-\theta)d\eta(\theta)\phi(\xi)d\xi,
\label{innerproduct}\end{equation}
where $\eta(\theta)=\eta(\theta,0)$. Then $A(0)$ and $A^*$ are adjoint operators. By the discussion in section \ref{stability}, we know that $\pm i\omega^*\tau^*$ are eigenvalues of $A(0)$. Thus, they are also eigenvalues of $A^*$. We need to compute the eigenvector of $A(0)$ and $A^*$ corresponding to $i\omega^*\tau^*$ and $-i\omega^*\tau^*$, respectively.
\par Suppose that $q(\theta)=(1,\alpha,\beta)^Te^{i\theta\omega^*\tau^*}$ is the eigenvector of $A(0)$ corresponding to $i\omega^*\tau^*$, then $A(0)q(\theta)=i\omega^*\tau^*q(\theta)$. It follows from the definition of $A(0)$ and $\eta(\theta,\mu)$ that
\[\tau^* \left(
           \begin{array}{ccc}
             i\omega^*+b_1x^* & c_1x^* & 0 \\
             -\alpha_1y^* & i\omega^* & c_2y^* \\
             0 & -\alpha_2z_2^*e^{-d_2\tau^*}e{-i\omega^*\tau^*} & i\omega^*-\alpha_2y^*e^{-d_2\tau^*}e{-i\omega^*\tau^*} \\
           \end{array}
         \right)q(0)=0.
 \]
Then, we can easily obtain
\[q(0)=(1,\alpha,\beta)^T,\]
where $\alpha=-\frac{i\omega^*+b_1 x^*}{c_1x^*}$ and $\beta=\frac{\alpha\alpha_2 z_2^* e^{-d_2\tau^*} e^{-i\omega^*\tau^*} }{i\omega^*+ d_3- \alpha_2 y^* e^{-d_2\tau^*} e^{-i\omega^*\tau^*}}$.\\
Similarly, let $q^*(\theta)=D(1,\alpha^*,\beta^*)e^{-i\theta\omega^*\tau^*}$  be the eigenvector of $A^*$ corresponding to  $-i\omega^*\tau^*$, then similarly we can obtain
\[\alpha^*=\frac{b_1x^*-i\omega^*}{\alpha_1y^*}, \ \ \beta^*=\frac{c_2y^*(-i\omega^*+b_1x^*)}{\alpha_1y^*(i\omega^*-d_3+\alpha_2y^*e^{-d_2\tau^*}e^{i\omega^*\tau^*})}. \]
By (\ref{innerproduct}) we get
\begin{eqnarray}
\langle q^*(s),q(\theta)\rangle &=& \bar{D}(1,\bar{\alpha}^*,\bar{\beta}^*)(1,\alpha,\beta)^T\nonumber\\
& &-\int_{-1}^0\int_{\xi=0}^{\theta}\bar{D}(1,\bar{\alpha}^*,\bar{\beta}^*)e^{-\omega^*\tau^*(\xi-\theta)}d\eta(\theta)(1,\alpha,\beta)^Te^{i\omega^*\tau^*\xi}d\xi\nonumber\\
&=& \bar{D}\left[1+\alpha\bar{\alpha}^*+\beta\bar{\beta}^*-(1,\bar{\alpha}^*,\bar{\beta}^*)\int_{-1}^0\phi(\theta)d\eta(\theta)(1,\alpha,\beta)^T  \right]\nonumber\\
& =& \bar{D}\left[ 1+\alpha\bar{\alpha}^*+\beta\bar{\beta}^*+\tau^*\left( \alpha z_2^*+\beta y^*\right)\alpha_2\bar{\beta}^*e^{-d_2\tau^*}e^{-i\omega^*\tau^*}\right]\nonumber.
\end{eqnarray}
Then we choose
\[ \bar{D}=\frac{1}{1+\alpha\bar{\alpha}^*+\beta\bar{\beta}^*+\tau^*\left( \alpha z_2^*+\beta y^*\right)\alpha_2\bar{\beta}^*e^{-d_2\tau^*}e^{-i\omega^*\tau^*}},\]
such that $\langle q^*(s),q(\theta)\rangle=1$ and $\langle q^*(s),\bar{q}(\theta)\rangle=0$.
\par In the following, we use the ideas in Hassard et al. \citep{HassardKazarinoff1981} to compute the coordinates describing center manifold $C_0$ at $\mu=0$. Define
\begin{equation}
z(t)=\langle q^*,x_t\rangle,\ W(t,\theta)=x_t(0)-2Re[z(t)q(\theta)],\label{Wttheta}\end{equation}
On the center manifold $C_0$, we have
\begin{equation}
W(t,\theta)=W(z(t),\bar{z}(t),\theta)=W_{20}(\theta)\frac{z^2}{2}+W_{11}(\theta)z\bar{z}+W_{02}(\theta)\frac{\bar{z}^2}{2}+\cdots \label{Wttheta2},
\end{equation}
where $z$ and $\bar{z}$ are local coordinates for $C_0$ in $C$ in the direction of $q^*$ and $\bar{q}^*$. Note that $W$ is real if $x_t$ is real. We deal only with the real solution. For solution $x_t\in C_0$ of (\ref{x_t}), since $\mu=0$, we have
\begin{eqnarray}
 \dot{z}(t)&=&i\omega^*\tau^*z+\bar{q}^*(0)F(0,W(z,\bar{z},0)+2Re\{zq(0)\})\nonumber\\
 &\stackrel{def}{=}& i \omega^*\tau^*+\bar{q}^*(0)F_0(z,\bar{z})=i\omega^*\tau^*+g(z,\bar{z}),\nonumber
\end{eqnarray}
where
\begin{equation}
g(z,\bar{z})=\bar{q}^*(0)F_0(z,\bar{z})=g_{20}(\theta)\frac{z^2}{2}+g_{11}(\theta)z\bar{z}+g_{02}(\theta)\frac{\bar{z}^2}{2}+\cdots.
\label{gzzbar}\end{equation}
From (\ref{Wttheta}) and (\ref{Wttheta2}), we have
\[ x_t(\theta) =(x_{1t}(\theta),x_{2t}(\theta),x_{3t}(\theta))=W(t,\theta)+zq(\theta)+\overline{zq(\theta)},\]
and
\[ q(\theta)=(1,\alpha,\beta)^Te^{i\theta\omega^*\tau^*}.\]
Thus, we can easily obtain that
\[ x_{1t}(0)=W_{20}^{(1)}(0)\frac{z^2}{2}+W_{11}^{(1)}(0)z\bar{z}+W_{02}^{(1)}(0)\frac{\bar{z}^2}{2}+z+\bar{z}+O\left(\left| (z,\bar{z})\right|^3 \right),\]
\[ x_{2t}(0)=W_{20}^{(1)}(0)\frac{z^2}{2}+W_{11}^{(1)}(0)z\bar{z}+W_{02}^{(1)}(0)\frac{\bar{z}^2}{2}+\alpha z+\overline{\alpha z}+O\left(\left| (z,\bar{z})\right|^3 \right),\]
\[ x_{3t}(0)=W_{20}^{(1)}(0)\frac{z^2}{2}+W_{11}^{(1)}(0)z\bar{z}+W_{02}^{(1)}(0)\frac{\bar{z}^2}{2}+\beta z+\overline{\beta z}+O\left(\left| (z,\bar{z})\right|^3 \right),\]
\begin{eqnarray}
x_{2t}(-1)=W_{20}^{(1)}(-1)\frac{z^2}{2}+W_{11}^{(1)}(-1)z\bar{z}&+&W_{02}^{(1)}(-1)\frac{\bar{z}^2}{2}+\alpha z e^{-i \omega^*\tau^*}\nonumber\\
 &+& \overline{\alpha z}e^{i \omega^*\tau^*}+ O\left(\left| (z,\bar{z})\right|^3 \right),\nonumber\\
x_{3t}(-1)=W_{20}^{(1)}(-1)\frac{z^2}{2}+W_{11}^{(1)}(-1)z\bar{z}&+&W_{02}^{(1)}(-1)\frac{\bar{z}^2}{2}+\beta z e^{-i \omega^*\tau^*}\nonumber\\
& + &\overline{\beta z}e^{i \omega^*\tau^*}+ O\left(\left| (z,\bar{z})\right|^3 \right).\nonumber
\end{eqnarray}
From the definition of $F(\mu,x_t)$, we have
\begin{eqnarray}
g(z,\bar{z})& = & \tau^*\bar{D}(1,\bar{\alpha}^*,\bar{\beta}^*)\left(
                         \begin{array}{c}
                           -b_1x_{1t}^2(0)-c_1x_{1t}(0)x_{2t}(0) \\
                           \alpha_1x_{1t}(0)x_{2t}(0)-c_2{x_2t}(0)x_{3t}(0) \\
                           \alpha_2e^{-d_2\tau^*}x_{2t}(-1)x_{3t}(-1) \\
                         \end{array}
                       \right)\nonumber\\
&=& \tau^* \bar{D}\left\{ z^2\left[ -b_1-\alpha(c_1-\alpha_1\bar{\alpha}^*)-c_2\alpha\beta\bar{\alpha}^*+\alpha\alpha_2\beta\bar{\beta}^* e^{-d_2\tau^*}e^{-2i\omega^*\tau^*}\right]\right.\nonumber\\
& & + 2z\bar{z}\left[-b_1-(c_1-\alpha_1\bar{\alpha}^*)Re\{\alpha\}-c_2Re\{\alpha\bar{\beta}\}+\alpha_2\bar{\beta}^* e^{-d_2\tau^*}Re\{\alpha\bar{\beta}\} \right]\nonumber\\
& & +\bar{z}^2\left[-b_1-(c_1-\alpha_1\bar{\alpha}^*)\bar{\alpha}-c_2\bar{\alpha}^*\bar{\alpha}\bar{\beta}+\alpha_2\bar{\beta}^*\bar{\alpha}\bar{\beta} e^{-d_2\tau^*}e^{2i\omega^*\tau^*} \right]\nonumber\\
& & + \frac{1}{2}z^2\bar{z}\left[-2b_1W_{20}^{(1)}(0)-4b_1W_{11}^{(1)}(0) \right.\nonumber\\
& & - (c_1-\alpha_1\bar{\alpha}^*)\left( \bar{\alpha}W_{20}^{(1)}(0)+2\alpha W_{11}^{(1)}(0)+2W_{11}^{(2)}(0)+W_{20}^{(2)}(0)\right)\nonumber\\
& & -c_2\bar{\alpha}^*\left( \bar{\beta}W_{20}^{(2)}(0)+2\beta W_{11}^{(2)}(0)+2\alpha W_{11}^{(3)}(0)+\bar{\alpha}W_{20}^{(3)}(0)\right)\nonumber\\
& & + \alpha_2 \bar{\beta}^* e^{-d_2\tau^*}\left( \bar{\beta}e^{i\omega^*\tau^*}W_{20}^{(2)}(-1)+2\beta e^{-i\omega^*\tau^*}W_{11}^{(2)}(-1)\right.\nonumber\\
& & + \left. \left.\left. 2\alpha e^{-i\omega^*\tau^*}W_{11}^{(3)}(-1)+\bar{\alpha}e^{i\omega^*\tau^*}W_{20}^{(3)}(-1)\right)\right]\right\}.\nonumber
\end{eqnarray}
Comparing the coefficients with (\ref{gzzbar}), we obtain
\[g_{20}=2\tau^*\bar{D}[-b_1-(c_1-\alpha_1\bar{\alpha}^*)\alpha-c_2\alpha\beta\bar{\alpha}^*+\alpha\alpha_2\beta\bar{\beta}^* e^{-d_2\tau^*}e^{-2i\omega^*\tau^*}], \]
\[g_{11}=2\tau^*\bar{D}[-b_1-(c_1-\alpha_1\bar{\alpha}^*)Re\{\alpha\}-c_2Re\{\alpha\bar{\beta}\}+\alpha_2\bar{\beta}^* e^{-d_2\tau^*}Re\{\alpha\bar{\beta}\}], \]
\[g_{02}=2\tau^*\bar{D}[-b_1-(c_1-\alpha_1\bar{\alpha}^*)\bar{\alpha}-c_2\bar{\alpha}^*\bar{\alpha}\bar{\beta}+\alpha_2\bar{\beta}^*\bar{\alpha}\bar{\beta} e^{-d_2\tau^*}e^{2i\omega^*\tau^*}], \]
\begin{eqnarray}
g_{21}& =& \tau^*\bar{D}\left[-b_1\left(2W_{20}^{(1)}(0)+4W_{11}^{(1)}(0)\right)\right.\nonumber\\
& & - (c_1-\alpha_1\bar{\alpha}^*)\left(\bar{\alpha}W_{20}^{(1)}(0)+2\alpha W_{11}^{(1)}(0)+2W_{11}^{(2)}(0)+W_{20}^{(2)}(0)\right)\nonumber\\
& & -c_2\bar{\alpha}^*\left( \bar{\beta}W_{20}^{(2)}(0)+2\beta W_{11}^{(2)}(0)+2\alpha W_{11}^{(3)}(0)+\bar{\alpha}W_{20}^{(3)}(0)\right)\nonumber\\
& & + \alpha_2 \bar{\beta}^* e^{-d_2\tau^*}\left( \bar{\beta}e^{i\omega^*\tau^*}W_{20}^{(2)}(-1)+2\beta e^{-i\omega^*\tau^*}W_{11}^{(2)}(-1)\right.\nonumber\\
& & + \left. \left. 2\alpha e^{-i\omega^*\tau^*}W_{11}^{(3)}(-1)+\bar{\alpha}e^{i\omega^*\tau^*}W_{20}^{(3)}(-1)\right)\right].\nonumber
\end{eqnarray}
In order to determine $g_{21}$, we need to compute $W_{20}(\theta)$ and $W_{11}(\theta)$. From (\ref{x_t}) and (\ref{Wttheta}), we have
\begin{eqnarray}
\dot{W}&=&\dot{x_t}-zq-\dot{\overline{zq}}=\left\{\begin{array}{ll}
                                                  AW-2Re[\bar{q}^*(0)F_0q(\theta)],  &\theta\in[-1,0),  \\[0.1in]
                                                  AW-2Re[\bar{q}^*(0)F_0q(\theta)]+F_0, & \theta=0,
                                                \end{array}
                                              \right.\nonumber\\
&\stackrel{def}{=}& AW+H(z,\bar{z},\theta),
\label{Wdot}
\end{eqnarray}
where
\begin{equation}
H(z,\bar{z},\theta)=H_{20}(\theta)\frac{z^2}{2}+H_{11}(\theta)z\bar{z}+H_{02}(\theta)\frac{\bar{z}^2}{2}+\cdots.
\label{Hzzbar}\end{equation}
Note that on center manifold $C_0$ near the origin
\[ \dot{W}=W_z\dot{z}+W_{\bar{z}}\dot{\bar{z}}.\]
Thus, we obtain
\begin{equation}
 (A-2i\omega^*\tau^*)W_{20}(\theta)=-H_{20}(\theta),\ AW_{11}(\theta)=-H_{11}(\theta).
\label{H20H11}\end{equation}
Comparing the coefficient with (\ref{Hzzbar}) gives that
\begin{equation}
H_{20}(\theta)=-g_{20}q(\theta)-\bar{g}_{02}\bar{q}(\theta), \ H_{11}(\theta)=-g_{11}q(\theta)-\bar{g}_{11}\bar{q}(\theta).
\label{H20H11other}\end{equation}
From (\ref{H20H11}), (\ref{H20H11other}) and the definition of $A$, we have
\[ \dot{W}_{20}(\theta)=2i\omega^*\tau^*W_{20}(\theta)+g_{20}q(\theta)+\bar{g}_{02}\bar{q}(\theta).\]
Noting $q(\theta)=q(0)e^{i\omega^*\tau^*\theta}$, hence
\begin{equation}
W_{20}(\theta)=\frac{ig_{20}}{\omega^*\tau^*}q(0)e^{i\omega^*\tau^*\theta}+\frac{i\bar{g}_{02}}{3\omega^*\tau^*}\bar{q}(0)e^{-2i\omega^*\tau^*\theta}+E_1e^{2i\omega^*\tau^*\theta},
\label{W20theta}\end{equation}
where $E_1=(E_1^{(1)},E_1^{(2)},E_1^{(3)})\in R^3$ is a constant vector.
Similarly, from (\ref{H20H11}) and (\ref{H20H11other}), we obtain
\begin{equation}
W_{11}(\theta)=-\frac{ig_{11}}{\omega^*\tau^*}q(0)e^{i\omega^*\tau^*\theta}+\frac{i\bar{g}_{11}}{\omega^*\tau^*}\bar{q}(0)e^{-i\omega^*\tau^*\theta}+E_2,\label{W11theta}\end{equation}
where $E_2=(E_2^{(1)},E_2^{(2)},E_2^{(3)})\in R^3$ is also a constant vector.
\par In the following we shall find out $E_1$ and $E_2$. From the definition of $A$ and (\ref{H20H11}), we can obtain
\begin{equation}
\int_{-1}^0d\eta(\theta)W_{20}(\theta)=2i\omega^*\tau^*W_{20}(\theta)-H_{20}(0),
\label{first}\end{equation}
and
\begin{equation}
\int_{-1}^0d\eta(\theta)W_{11}(\theta)=-H_{11}(0),
\label{second}\end{equation}
where $\eta(\theta)=\eta(0,\theta)$. From (\ref{Wdot}) and (\ref{Hzzbar}), we have
\begin{equation}
H_{20}(0)=-g_{20}q(0)-\bar{g}_{02}\bar{q}(0)+2\tau^*\left(
                                                      \begin{array}{c}
                                                        -b_1-c_1\alpha \\
                                                        \alpha\alpha_1-c_2\alpha\beta \\
                                                        \alpha_2\alpha\beta e^{-d_2\tau^*}e^{-2i\omega^*\tau^*} \\
                                                      \end{array}
                                                    \right),
\label{H20theta}\end{equation}
and
\begin{equation}
H_{11}(0)=-g_{11}q(0)-\bar{g}_{11}\bar{q}(0)+2\tau^*\left(
                                                      \begin{array}{c}
                                                        -b_1-c_1Re\{\alpha\} \\
                                                        \alpha_1Re\{\alpha\}-c_2Re\{\alpha\bar{\beta}\} \\
                                                        \alpha_2e^{-d_2\tau^*}Re\{\alpha\bar{\beta} \\
                                                      \end{array}
                                                    \right).
\label{H11theta}\end{equation}
Substituting (\ref{W20theta}) and (\ref{W11theta}) into (\ref{first}) and noticing that
\[\left(i\omega^*\tau^*I-\int_{-1}^0e^{i\omega^*\tau^*\theta}d\eta(\theta)\right)q(0)=0, \]
and
\[ \left(-i\omega^*\tau^*I-\int_{-1}^0e^{-i\omega^*\tau^*\theta}d\eta(\theta)\right)\bar{q}(0)=0, \]
we obtain
\[\left(2i\omega^*\tau^*I-\int_{-1}^0e^{2i\omega^*\tau^*\theta}d\eta(\theta)\right)E_1= 2\tau^*\left(
                                                      \begin{array}{c}
                                                        -b_1-c_1\alpha \\
                                                        \alpha\alpha_1-c_2\alpha\beta \\
                                                        \alpha_2\alpha\beta e^{-d_2\tau^*}e^{-2i\omega^*\tau^*} \\
                                                      \end{array}
                                                    \right),
\]
which leads to
\[\left(
    \begin{array}{ccc}
      2i\omega^*+b_1x^* & c_1x^* & 0 \\
      -\alpha_1y^* & 2i\omega^* & c_2y^* \\
      0 & -\alpha_2z_2^*e^{-d_2\tau^*}e^{-2i\omega^*\tau^*} & \alpha_2y^*e^{-d_2\tau^*}e^{-2i\omega^*\tau^*} \\
    \end{array}
  \right)
E_1\]
\[\ \ \ =2\left(
                                                      \begin{array}{c}
                                                        -b_1-c_1\alpha \\
                                                        \alpha\alpha_1-c_2\alpha\beta \\
                                                        \alpha_2\alpha\beta e^{-d_2\tau^*}e^{-2i\omega^*\tau^*} \\
                                                      \end{array}
                                                    \right).
\]
It follows that
\begin{eqnarray}
E_1&=& 2{\left(
    \begin{array}{ccc}
      2i\omega^*+b_1x^* & c_1x^* & 0 \\
      -\alpha_1y^* & 2i\omega^* & c_2y^* \\
      0 & -\alpha_2z_2^*e^{-d_2\tau^*}e^{-2i\omega^*\tau^*} & \beta_2-\alpha_2y^*e^{-d_2\tau^*}e^{-2i\omega^*\tau^*} \\
    \end{array}
  \right)
}^{-1}\nonumber\\
& & \times \left(
                                                      \begin{array}{c}
                                                        -b_1-c_1\alpha \\
                                                        \alpha\alpha_1-c_2\alpha\beta \\
                                                        \alpha_2\alpha\beta e^{-d_2\tau^*}e^{-2i\omega^*\tau^*} \\
                                                      \end{array}
                                                    \right).\nonumber
\end{eqnarray}
Similarly, substituting (\ref{W11theta}) and (\ref{H11theta}) into (\ref{second}), we can get
\begin{equation}
E_2 = 2{\left(
    \begin{array}{ccc}
      b_1x^* & c_1x^* & 0 \\
      -\alpha_1y^* & 0 & c_2y^* \\
      0 & -\alpha_2z_2^*e^{-d_2\tau^*} & \beta_2-\alpha_2y^*e^{-d_2\tau^*} \\
    \end{array}
  \right)
}^{-1}
\times\left(
                                                      \begin{array}{c}
                                                        -b_1-c_1Re\{\alpha\} \\
                                                        \alpha_1Re\{\alpha\}-c_2Re\{\alpha\bar{\beta}\} \\
                                                        \alpha_2e^{-d_2\tau^*}Re\{\alpha\bar{\beta}\} \\
                                                      \end{array}
                                                    \right).
                                                    \end{equation}
Thus, we can determine $W_{20}$ and $W_{11}$ from (\ref{W20theta}) and (\ref{W11theta}).

%
\end{document}